\documentclass[12pt]{amsart}
\usepackage{color,graphics,graphicx,shortvrb} 
\usepackage{epsfig}
\usepackage{amssymb,amsmath,amsfonts}
\usepackage{newlfont}
\usepackage{caption}
\usepackage{latexsym}

\usepackage[utf8]{inputenc}
\usepackage{subcaption}

\def\vbar{\mathchoice{\vrule height6.3ptdepth-.5ptwidth.8pt\kern-.8pt}
   {\vrule height6.3ptdepth-.5ptwidth.8pt\kern-.8pt}
   {\vrule height4.1ptdepth-.35ptwidth.6pt\kern-.6pt}
   {\vrule height3.1ptdepth-.25ptwidth.5pt\kern-.5pt}}
\def\fudge{\mathchoice{}{}{\mkern.5mu}{\mkern.8mu}}
\def\bbc#1#2{{\rm \mkern#2mu\vbar\mkern-#2mu#1}}
\def\bbb#1{{\rm I\mkern-3.5mu #1}}
\def\bba#1#2{{\rm #1\mkern-#2mu\fudge #1}}
\def\bb#1{{\count4=`#1 \advance\count4by-64 \ifcase\count4\or\bba
A{11.5}\or \bbb B\or\bbc C{5}\or\bbb D\or\bbb E\or\bbb F \or\bbc
G{5}\or\bbb H\or \bbb I\or\bbc J{3}\or\bbb K\or\bbb L \or\bbb
M\or\bbb N\or\bbc O{5} \or \bbb P\or\bbc Q{5}\or\bbb R\or\bbc
S{4.2}\or\bba T{10.5}\or\bbc U{5}\or \bba V{12}\or\bba
W{16.5}\or\bba X{11}\or\bba Y{11.7}\or\bba Z{7.5}\fi}}
\renewcommand{\Re}{\mathrm{Re\,}}

\def \n {\noindent}
\def \barr {\begin{array}{l}} 
\def \ear {\end{array}}
\def \beq {\begin{equation}} 
\def \eeq {\end{equation}}
\def \beqn {\begin{eqnarray}}
\def \eeqn {\end{eqnarray}}
\def \f {\end{document}}

\def\dfrac{\displaystyle\frac}

\newtheorem{lem}{Lemma}
\newtheorem{rem}{Remark}
\newtheorem{prop}{Proposition}

\newtheorem{theo}{Theorem}

\newcommand{\field}[1]{\mathbb{#1}}
\newcommand{\R}{\field{R}}

\newcommand{\N}{\field{N}}

\numberwithin{equation}{section}

\begin{document}

\title[Asymptotic analysis of a 2D overhead crane with input delays]{Asymptotic analysis of a 2D overhead crane with input delays in the boundary control}

\author{Fadhel Al-Musallam}
\address{Kuwait University, Faculty of Science, Department of Mathematics, Safat 13060, Kuwait}
\email{musallam@sci.kuniv.edu.kw}

\author{Ka\"{\i}s Ammari}
\address{UR Analysis and Control of PDEs, UR13ES64, Department of Mathematics, Faculty of Sciences of Monastir, University of Monastir, 5019 Monastir, Tunisia}
\email{kais.ammari@fsm.rnu.tn}

\author{Boumedi\`ene Chentouf}
\address{Kuwait University, Faculty of Science, Department of Mathematics, Safat 13060, Kuwait}
\email{chenboum@hotmail.com,chentouf@sci.kuniv.edu.kw}

\begin{abstract} 
The paper investigates the asymptotic behavior of a 2D overhead crane with
input delays in the boundary control. A linear boundary control is proposed.
The main feature of such a control lies in the facts that it solely depends on
the velocity but under the presence of time-delays. We end-up with a
closed-loop system where no displacement term is involved. It is shown that
the problem is well-posed in the sense of semigroups theory. LaSalle's
invariance principle is invoked in order to establish the asymptotic
convergence for the solutions of the system to a stationary position which
depends on the initial data. Using a resolvent method it is proved that the
convergence is indeed polynomial.
\end{abstract}   


\maketitle
\tableofcontents

\thispagestyle{empty}


\section{Introduction}

\setcounter{equation}{0}

Overhead cranes are extensively utilized in a variety of industrial and
construction sites. Usually, it consists of a hoisting mechanism such as a
hoisting cable and a hook and a support mechanism like a girder (trolley) \cite{anm}. The
aim of using such cranes is to horizontally transport point-to-point a
suspended mass/load. It is well-known that cables possess the inherent
flexibility characteristics and can only develop tension \cite{anm}. Such natural
features inevitably cause deflection in transversal direction of the cable.
Furthermore, the suspended load is always subject to swings due to several
reasons. Thereby, the behavior of the overhead crane system with flexible
cable can generate complex system dynamics (see \cite{anm} for more details).

\noindent We shall consider in the present work an overhead crane system which
consists of a motorized platform of mass $m$ moving along an horizontal rail.
A flexible cable of length $\ell$, holding a load mass $M$, is attached to the
platform (see Fig. \ref{fi1}). Furthermore, it is assumed that:

\begin{quote}
(i) The cable is completely flexible and non-stretching.
\newline(ii) The
length of the cable is constant.
\newline(iii) Transversal and angular
displacements are small.
\newline(iv) Friction is neglected.
\newline(v) The
masses $m$ and $M$ are point masses.
\newline(vi) The angle of the cable with
respect to the vertical $x$-axis is small everywhere.
\end{quote}

\begin{figure}
\centering
\includegraphics[scale=0.5]{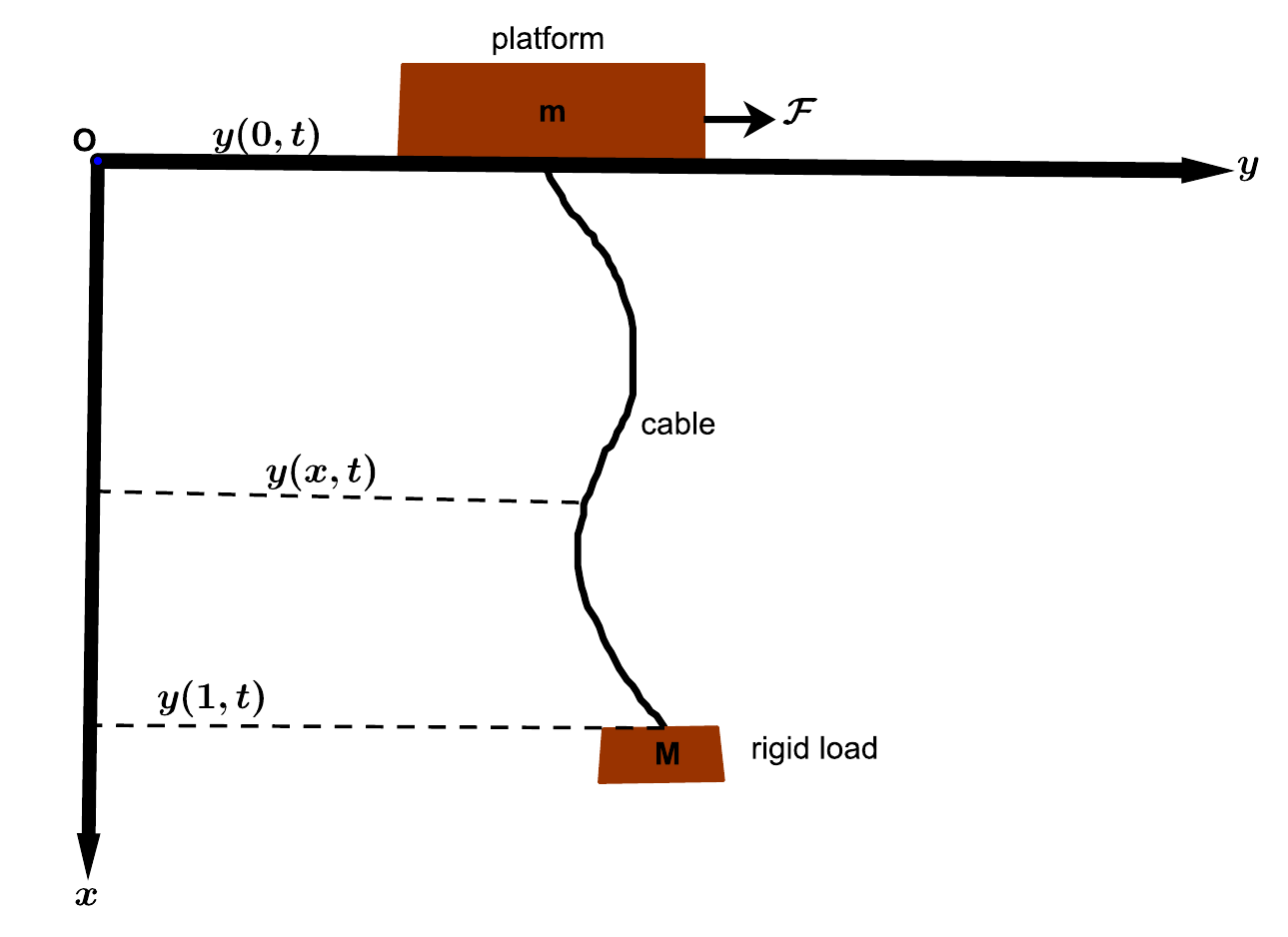}
\caption{The overhead crane model}
\label{fi1}      
\end{figure}

\noindent Under the above assumptions, the overhead crane is modeled by a
hybrid PDE-ODE system (see \cite{ANBCR} and \cite{Ra}). For sake of
completeness, we shall provide some details about the derivation of such a
model (the reader is referred to \cite{ANBCR} and \cite{Ra} for more details).

Let $T$ be the the tension of the cable, $\theta(x,t)$ be the angle between
$T$ and the $x$-axis, and consider a portion of the cable of length $\Delta x
$. Newton's law leads to%

\[
\Delta xy_{tt}(x,t)=T(x+\Delta x)\theta(x+\Delta x,t)-T(x)\theta(x,t).
\]
We can write $\theta(x,t)\simeq y_{x}(x,t)$ due to the assumption of smallness
of transversal and angular displacements. On the other hand, since the tension
of the cable is essentially due to the action on its lower part, we have
$|T(x)|=(M+\ell-x)g$, which is the modulus of tension of the cable and will be
denoted by $a(x)$. This, together with the above equation imply that
\begin{equation}%
\begin{array}
[c]{ll}
y_{tt}(x,t)-\left(  ay_{x}\right)  _{x}(x,t)=0, & 0<x<\ell,\;t>0.
\end{array}
\label{n1}%
\end{equation}

We turn now to the equation of the platform part of the system (see Fig. 2).
Taking into account the external controlling force ${\mathcal{F}}(t)$, we
have
\[
m y_{tt}(0,t) =|T(0)| \theta(0,t)+{\mathcal{F}}(t),
\]
which can be rewritten
\begin{equation}%
\begin{array}
[c]{ll}%
m y_{tt}(0,t) =a(0) y_{x}(0,t)+{\mathcal{F}}(t), & t>0,
\end{array}
\label{n2}%
\end{equation}
as $|T(x)|=a(x)$ and $\theta(0,t) \simeq y_{x}(0,t)$.

\noindent Using similar arguments for the the load mass (see Fig. \ref{fi3}),
we have
\begin{equation}%
\begin{array}
[c]{ll}%
My_{tt}(\ell,t)=-a(\ell)y_{x}(\ell,t), & t>0.
\end{array}
\label{n3}%
\end{equation}
Combining (\ref{n1})-(\ref{n3}), we have the system%
\begin{equation}
\left\{
\begin{array}
[c]{ll}%
y_{tt}(x,t)-\left(  ay_{x}\right)  _{x}(x,t)=0, & 0<x<\ell,\;t>0,\\
my_{tt}(0,t)-\left(  ay_{x}\right)  (0,t)={\mathcal{F}}(t), & t>0,\\
My_{tt}(\ell,t)+\left(  ay_{x}\right)  (\ell,t)=0, & t>0,\\
&
\end{array}
\right. \label{(1.1)}%
\end{equation}
where $a(x)$ is supposed to satisfy the following conditions

\begin{equation}
\left\{
\begin{array}
[c]{l}%
a\in H^{1}(0,\ell);\\
\text{there exists a positive constant}\;a_{0}\;\text{such that}\;a(x) \geq
a_{0}>0 \;\;\text{for all}\;\;x\in\lbrack0,\ell].
\end{array}
\right. 
\label{1.3}
\end{equation}

For simplicity and without loss of generality, we shall set the length $\ell=1$.

\begin{figure}
\centering
\includegraphics[scale=0.6]{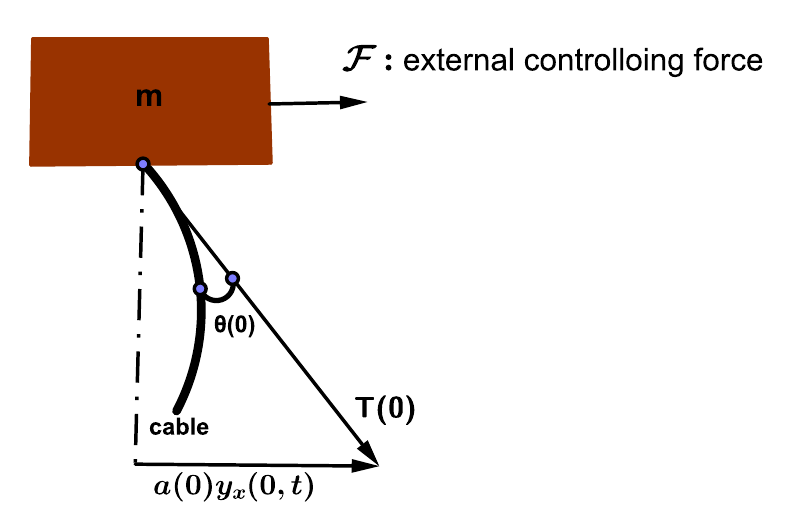}
\caption{The platform} 
\label{fi2}    
\end{figure} 

\begin{figure}
\centering
\includegraphics[scale=0.7]{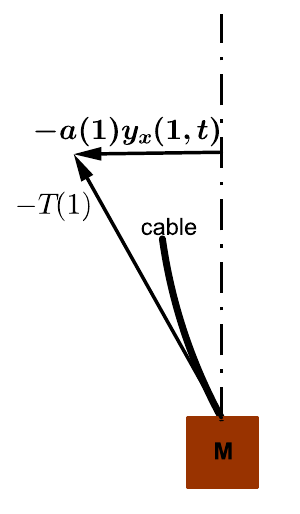} 
\caption{The load mass} 
\label{fi3}
\end{figure}

\n As mentioned above, the objective is to seek a delayed control ${\mathcal F}(t)$ depending solely on the velocity so that the solutions of the closed-loop system asymptotically converge to an equilibrium point in a suitable functional space.

The boundary stabilization of the system (\ref{(1.1)}) has been the object of
a considerable mathematical research. There are two categories of research
articles: in the first category, at least one of the dynamical terms in the
boundary conditions is neglected. In other words, either $my_{tt}(0,t)$ or
$My_{tt}(1,t)$ does not appear in the system or even both terms are not
present. For instance, it has been shown in \cite{Ra} that the feedback law
\[
{\mathcal{F}}(t)=-c y(0,t)-F(y_{t}(0,t)),\;c>0,
\]
exponentially stabilizes the system (\ref{(1.1)}) with $my_{tt}(0,t)=0$ under
appropriate assumptions on the function $F$. Another stabilization result for
the system (\ref{(1.1)}) with $my_{tt}(0,t)=My_{tt}(1,t)=0$ has also been
established in \cite{CR} via the action of the following feedback:
\[
\left\{
\begin{array}
[c]{l}%
{\mathcal{F}}(t)=-\alpha y(0,t)-F(y_{t}(0,t)),\\
{\mathcal{U}}(t)=-\alpha y(1,t)-F(y_{t}(1,t)),\;\;\alpha>0,
\end{array}
\right.
\]
where ${\mathcal{U}}$ is an additional control to be applied on the load mass.
In \cite{ANBCR}, the asymptotic stabilization has been proved as long as a
dynamical control is acting on the boundary $y(1,t)$. We also mention that a
stabilization result has been obtained in \cite{Cher} by proposing the
feedback law
\[
{\mathcal{F}}(t)=k_{p}y(0,t)+k_{v}y_{t}(0,t)+\displaystyle\int_{0}%
^{1}G(x)y(x,t)dx+\frac{k_{v}}{k_{p}}\int_{0}^{1}G(x)y_{t}(x,t)dx,
\]
with $k_{p},\;k_{v}>0$ and $G$ is a function in $H^{1}(0,1)$. Of course, such
a result has been established under some conditions on the feedback gains
$k_{p},\;k_{v}$ as well as the function $G$. Similar findings have been
obtained in \cite{AC} for other types of controls containing a displacement
term. We conclude this discussion about the first category of articles
available in the literature by pointing out that it has been noticed in
\cite{bc01} that in all references cited above, either the boundary conditions
in (\ref{(1.1)}) or the stabilizing feedback law ${\mathcal{F}}(t)$ involves
the displacement term $y$. This is mainly due to the fact that most of the
authors defined the energy-norm of the system by
\[
E_{0}(t)=\frac{1}{2}\int_{0}^{1}\left(  y_{x}^{2}+y_{t}^{2}\right)  dx.
\]
This observation has motivated the authors in \cite{bc01} to consider a
displacement term in the equation and propose a general class of feedback law
containing only the velocity. In fact, the closed-loop system in \cite{bc01}
has the following form
\begin{equation}
\left\{
\begin{array}
[c]{ll}%
y_{tt}(x,t)-\left(  ay_{x}\right)  _{x}(x,t)+\alpha y_{t}(x,t)+\beta
y(x,t)=0,\;\;\alpha\geq0,\;\beta>0, & 0<x<1,\;t>0,\\
\left(  ay_{x}\right)  (0,t)=\epsilon_{1}f(y_{t}(0,t)), & t>0,\\
\left(  ay_{x}\right)  (1,t)=\epsilon_{2}g(y_{t}(1,t)), & t>0,\label{(1.6)}%
\end{array}
\right.
\end{equation}
in which $f$ and $g$ are two nonlinear functions. The multiplier method has
been successfully used in \cite{bc01} to get precise decay rate (polynomial or
exponential) estimates of the energy of the system (\ref{(1.6)}) according to
the type of assumptions on the functions $f$ and $g$. Recently, the
back-stepping approach has been successfully applied to a variant of the
system (\ref{(1.1)}) leading to an exponentially stabilizing boundary feedback
controller \cite{AC}. In the same spirit, the following feedback law
\[
\left\{
\begin{array}
[c]{l}%
{\mathcal{F}}(t)=-\alpha_{1}y(0,t)-\beta_{1}y_{t}(0,t),\\
{\mathcal{U}}(t)=-\alpha_{2}y(1,t)-\beta_{2}y_{t}(1,t),
\end{array}
\right.
\]
has been suggested in \cite{gx} in the case where $\alpha_{1}+\alpha_{2}%
\not =0,\;\beta_{1}+\beta_{2}\not =0$ and $my_{tt}(0,t)=My_{tt}(1,t)=0$ and the
Riesz basis property has been shown.

The second category of research papers takes into consideration the dynamics
of both the load mass and platform mass. Within this context, it has been
proved in \cite{mif2} that the system (\ref{(1.1)}) can be strongly (but
non-uniformly) stabilized by means of the control
\[
{\mathcal{F}}(t)=-\alpha y(0,t)-f(y_{t}(0,t)),\;\alpha>0,
\]
where $f$ is a suitable function. This motivated several authors to
propose controls of higher orders to reach the uniform exponential stability.
Indeed, the uniform stabilization holds if
\[
{\mathcal{F}}(t)=-\alpha y(0,t)-\alpha\beta y_{t}(0,t)+\beta y_{xt}%
(0,t),\;\alpha>0,\,\beta>0.
\]
It turned out that the same result result can be achieved by the control
\[
{\mathcal{F}}(t)=-\alpha y(0,t)-(\beta+\alpha c )y_{t}(0,t)+c 
y_{xt}(0,t),
\]
where $\alpha,\beta$ and $c$ are positive constants satisfying
$\beta c < m.$ Motivated by the work of \cite{bc01}, a feedback control
depending only on the velocity has been proposed in \cite{cos} for the system
(\ref{(1.1)}) and an asymptotic convergence result has been established (see
also \cite{KZD}).

All the papers mentioned above do not take into consideration time-delay. In
turn, it is well-known that delays are inevitable in practice as they
naturally arises in most systems due to the time factor needed for the
communication among the controllers, the sensors and the actuators of systems
or in some cases due to the dependence of the state variables on past states.
Furthermore, it has been noticed that the presence of a delay in a system
could be a source of poor performance and instability \cite{d1}-\cite{d3} (see
also \cite{sha}\cite{KLSM}, \cite{KS} and \cite{KSe}).

\noindent The present work places primary emphasis on the analysis of the
system (\ref{(1.1)}) under the action of the following input delay 
\begin{equation}
{\mathcal F}(t)=-\beta y_t(0,t)+\alpha y_t (0,t-\tau),
\label{F}
\end{equation}
where $\beta>0$, $\alpha \in \R$ and $\tau >0$ is the time-delay. 

\noindent It is worth mentioning that the absence of the displacement term in
the closed-loop system prevents the applicability of classical Poincar\'{e}
inequalities. To overcome this difficulty, an appropriate energy-norm is suggested.

\noindent The main contribution of the present work is threefold:

\begin{enumerate}
\item[(a)] Extend the mathematical findings on the overhead crane available in
literature (specially those of \cite{mif1,mif2,bc01,cos}), where no delay has
been taken into account in the feedback laws.

\item[(b)] Show that despite the presence of the delay term in the proposed
feedback control law, the closed-loop system possesses the asymptotic
convergence property of its solutions to an equilibrium state which depends on
the initial conditions.

\item[(c)] Provide the rate of convergence of solutions of the closed-loop
system to the equilibrium state, in contrast to the work \cite{cos} where such
a result has not been achieved.
\end{enumerate}

The paper is organized as follows. The next section is devoted to the proof of
existence and uniqueness of the solutions to the closed-loop system. Section
\ref{sect3} deals with the asymptotic behavior of solutions via the use of
LaSalle's principle. Section \ref{sect4} is devoted to the polynomial
convergence of solutions. Finally, the paper closes with conclusions and discussions.


\section{Well-posedness of the system} \label{sect2}
\setcounter{equation}{0}

With the feedback law in (\ref{F}), we obtain the closed-loop system
\begin{equation}
\left\{
\begin{array}
[c]{ll}%
y_{tt}(x,t)-\left(  ay_{x}\right)  _{x}(x,t)=0, & 0<x<1,\;t>0,\\
my_{tt}(0,t)-\left(  ay_{x}\right)  (0,t)=-\beta y_{t}(0,t)+\alpha
y_{t}(0,t-\tau), & t>0,\\
My_{tt}(1,t)+\left(  ay_{x}\right)  (1,t)=0, & t>0,\label{1.6}%
\end{array}
\right.
\end{equation}
where $a$ obeys the condition (\ref{1.3}), $\alpha\in\mathbb{R}$ and $\beta>0$.

\noindent Our immediate task is to seek an appropriate energy associated to
(\ref{1.6}). To proceed, let
\begin{equation}
E_{0}(t)=\frac{1}{2}\left\{  \int_{0}^{1}\left(  y_{t}^{2}(x,t)+a(x)y_{x}%
^{2}(x,t)\right)  dx+my_{t}^{2}(0,t)+My_{t}^{2}(1,t)+K\tau\int_{0}^{1}%
y_{t}^{2}(0,t-x\tau)\;dx\right\}  ,\label{e0}%
\end{equation}
where $K$ is a positive constant. Using (\ref{1.6}) and integrating by parts,
a formal computation yields
\begin{equation}
E_{0}^{\prime}(t)=-\beta y_{t}^{2}(0,t)+\alpha y_{t}(0,t)y_{t}(0,t-\tau
)-\dfrac{K}{2}\left(  y_{t}^{2}(0,t-\tau)-y_{t}^{2}(0,t)\right)
.\label{(1.61)}%
\end{equation}
Applying Young's inequality, the latter becomes
\begin{equation}
E_{0}^{\prime}(t)\leq\left(  \dfrac{K}{2}+\dfrac{|\alpha|}{2c}-\beta\right)
y_{t}^{2}(0,t)+\dfrac{1}{2}\left(  |\alpha|c-K\right)  y_{t}^{2}%
(0,t-\tau),\label{1.6a}%
\end{equation}
for any positive constant $c$. Subsequently, we introduce the following
additional energy functional
\begin{equation}
E_{1}(t)=\frac{1}{2}\rho^{2}(t),\label{e1}%
\end{equation}
where
\begin{equation}
\rho(t)=\int_{0}^{1}y_{t}(x,t)dx+c_{1}y_{t}(0,t)+c_{2}y_{t}(1,t)+c_{3}\int
_{0}^{1}y_{t}(0,t-x\tau)\;dx+c_{4}y(0,t),\label{rho}%
\end{equation}
and $c_{1},c_{2},c_{3},$ and $c_{4}$ are constants to be determined. Following
the same arguments as for $E_{0}(t)$, we get
\begin{align}
E_{1}^{\prime}(t)  & =\rho(t)\Bigl((ay_{x})(1,t)\left[  1-\dfrac{c_{2}}%
{M}\right]  +(ay_{x})(0,t)\left[  \displaystyle\frac{c_{1}}{m}-1\right]
+\beta y_{t}(0,t)\left[  \displaystyle\frac{c_{3}}{\tau\beta}+\dfrac{c_{4}%
}{\beta}-\dfrac{c_{1}}{m}\right]  \Bigr.\nonumber\\
& \Bigl.\hspace{1.3cm}+\alpha y_{t}(0,t-\tau)\left[  \displaystyle\frac{c_{1}%
}{m}-\dfrac{c_{3}}{\tau\alpha}\right]  \Bigr).\label{e12}%
\end{align}
Thereafter, we define the total energy of the system (\ref{1.6}) as follows
\begin{equation}
{\mathcal{E}}(t)=E_{0}(t)+E_{1}(t).\label{ene}%
\end{equation}
This, together with (\ref{1.6a}) and (\ref{e12}), imply that
\begin{align}
{\mathcal{E}}^{\prime}(t)  & \leq\left(  \displaystyle\frac{K}{2}%
+\dfrac{|\alpha|}{2c}-\beta\right)  y_{t}^{2}(0,t)+\dfrac{1}{2}\left(
|\alpha|c-K\right)  y_{t}^{2}(0,t-\tau)\nonumber\\
& +\rho(t)\Bigl\{(ay_{x})(1,t)\left[  1-\dfrac{c_{2}}{M}\right]
+(ay_{x})(0,t)\left[  \displaystyle\frac{c_{1}}{m}-1\right]  +\beta
y_{t}(0,t)\left[  \displaystyle\frac{c_{3}}{\tau\beta}+\dfrac{c_{4}}{\beta
}-\dfrac{c_{1}}{m}\right]  \Bigr.\nonumber\\
& \Bigl.\hspace{1.3cm}+\alpha y_{t}(0,t-\tau)\left[  \displaystyle\frac{c_{1}%
}{m}-\dfrac{c_{3}}{\tau\alpha}\right]  \Bigr\}.\label{dec}%
\end{align}
In order to make the energy ${\mathcal{E}}(t)$ decreasing, we shall assume
that
\begin{equation}
|\alpha|<\beta,\label{sma}%
\end{equation}
and then choose $K$ such that
\begin{equation}
|\alpha|\leq K\leq 2\beta-|\alpha|,\label{uni}%
\end{equation}
whereas the other constants are
\begin{equation}
\left\{
\begin{array}
[c]{l}%
c=1,\\
c_{1}=m,\;c_{2}=M,\;c_{3}=\tau\alpha,\;c_{4}=\beta-\alpha.
\end{array}
\right. \label{uni2}%
\end{equation}
In light of (\ref{dec}) and (\ref{sma})-(\ref{uni2}), we deduce that
\begin{equation}
{\mathcal{E}}^{\prime}(t)\leq\displaystyle\frac{1}{2}\left\{  (-2\beta
+|\alpha|+K)y_{t}^{2}(0,t)+(|\alpha|-K)y_{t}^{2}(0,t-\tau)\right\}
\leq0,\label{sm}%
\end{equation}
and hence the energy ${\mathcal{E}}(t)$ is decreasing.

\begin{rem}
It is clear from the above choices in (\ref{uni2}), that the additional energy
$E_{1}(t)$ defined by (\ref{e1})-(\ref{rho}) is in fact constant. \label{zer}
\end{rem}

Here and elsewhere throughout the paper, we shall use the following definitions and notations for the Hilbert space $L^{2}(0,1)$ and the Sobolev space $H^{m}(0,1)$, more precisely
$$L^2 (0,1)= \left \{  v: (0,1) \rightarrow \R \;  \text{is measurable and} \; \int_0^{1} | v (x)|^2 \; dx < \infty \right \}$$
equipped  with its usual norm 
$$\| \varphi \|_{L^2 (0,1)}= \displaystyle \left( \int_0^{1} | v(x)|^2 \; dx \right)^{1/2},$$
and
$$H^m (0,1)= \left \{ g: (0,1) \rightarrow \R;  \; {g}^{(m)} \in  L^2 (0,1), \text{for} \; m \in \N \right \}$$
endowed with the standard norm
$$\| g  \|_{H^m (0,1)} = \displaystyle \sum_{i=0}^{i=m}  \displaystyle \| g^{(i)} \|_{L^2 (0,1)}.$$

Let us return now to our closed-loop system (\ref{1.6}). Using the well-known
change of variables \cite{da}
\begin{equation}
u(x,t)=y_{t}(0,t-x\tau),\label{ch}%
\end{equation}
the system (\ref{1.6}) becomes
\begin{equation}
\left\{
\begin{array}
[c]{ll}%
y_{tt}(x,t)-(ay_{x})_{x}(x,t)=0, & (x,t)\in(0,1)\times(0,\infty),\\
\tau u_{t}(x,t)+u_{x}(x,t)=0, & (x,t)\in(0,1)\times(0,\infty),\\[1mm]%
my_{tt}(0,t)-\left(  ay_{x}\right)  (0,t)=\alpha u(1,t)-\beta u(0,t), & t>0,\\
My_{tt}(1,t)+\left(  ay_{x}\right)  (1,t)=0, & t>0,\\
y(x,0)=y_{0}(x),\;y_{t}(x,0)=y_{1}(x), & x\in(0,1),\\
u(x,0)=y_{t}(0,-x\tau)=f(-x\tau), & x\in(0,1).
\end{array}
\right. \label{3}%
\end{equation}
Let $z(\cdot,t)=y_{t}(\cdot,t),\;\xi=y_{t}(0,t),\;\eta=y_{t}(1,t)$ and
consider the state variable $\Phi=(y,z,u,\xi,\eta).$ Then, our state space
${\mathcal{X}}$ is defined by
\[
{\mathcal{X}}=H^{1}(0,1)\times L^{2}(0,1)\times L^{2}(0,1)\times\mathbb{R}%
^{2},
\]
equipped with the following real inner product (the complex case is similar)
\begin{equation}%
\begin{array}
[c]{l}%
\langle(y,z,u,\xi,\eta),(\tilde{y},\tilde{z},\tilde{u},\tilde{\xi},\tilde
{\eta})\rangle_{\mathcal{X}}=\displaystyle\int_{0}^{1}\left(  ay_{x}\tilde
{y}_{x}+z\tilde{z}\right)  dx+K\tau\int_{0}^{1}u\tilde{u}\,dx+m\xi\tilde{\xi
}+M\eta\tilde{\eta}\\
+\displaystyle\varpi\left(  \int_{0}^{1}zdx+m\xi+M\eta+\mu y(0)+\tau
\alpha\int_{0}^{1}udx\right)  \left(  \int_{0}^{1}\tilde{z}dx+m\tilde{\xi
}+M\tilde{\eta}+\mu{\tilde{y}}(0)+\tau\alpha\int_{0}^{1}\tilde{u}dx\right)
\label{ip}%
\end{array}
\end{equation}
in which $K>0$ satisfies the condition (\ref{uni}), while $\mu=\beta
-\alpha$ and $\varpi$ is a positive constant to be determined. Note that
$\mu=\beta-\alpha$ is positive due to (\ref{sma}).

The first result is stated below.

\begin{prop}
Assume that (\ref{1.3}), (\ref{sma}) and (\ref{uni}) hold. Then, the state
space ${\mathcal{X}}$ endowed with the inner product (\ref{ip}) is a Hilbert
space provided that $\varpi$ is small enough. \label{p1}
\end{prop}

\begin{proof}
It suffices to show the existence of two positive constants $A_{1}$ and
$A_{2}$ such that
\begin{equation}
A_{1}\Vert(y,z,u,\xi,\eta)\Vert\leq\Vert(y,z,u,\xi,\eta)\Vert
_{\scriptscriptstyle{\mathcal{X}}}\leq A_{2}\Vert(y,z,u,\xi,\eta
)\Vert,\label{in}%
\end{equation}
where $\Vert(y,z,u,\xi,\eta)\Vert$ denotes the usual norm of $H^{1}(0,1)\times
L^{2}(0,1)\times L^{2}(0,1)\times\mathbb{R}^{2}$, that is,
\[
\Vert(y,z,u,\xi,\eta)\Vert^{2}=\displaystyle\int_{0}^{1}\left(  y^{2}%
+y_{x}^{2}+z^{2}+u^{2}\right)  dx+\xi^{2}+\eta^{2}.
\]
The right-hand inequality $\Vert(y,z,u,\xi,\eta)\Vert
_{\scriptscriptstyle{\mathcal{X}}}\leq A_{2}\Vert(y,z,u,\xi,\eta)\Vert$ is
straightforward. Indeed, Young's and H\"{o}lder's inequalities yield
\[%
\begin{array}
[c]{l}%
\Vert(y,z,u,\xi,\eta)\Vert_{\scriptscriptstyle{\mathcal{X}}}^{2}%
\leq\displaystyle\int_{0}^{1}\left(  ay_{x}^{2}+z^{2}\right)  dx+K\tau\int
_{0}^{1}u^{2}\,dx+m\xi^{2}+M\eta^{2}\\
+\displaystyle5\varpi\left(  \int_{0}^{1}z^{2}dx+m^{2}\xi^{2}+M^{2}\eta
^{2}+\mu^{2}y^{2}(0)+\tau^{2}\alpha^{2}\int_{0}^{1}u^{2}dx\right)  .
\end{array}
\]
Moreover, by virtue of (\ref{1.3}) and the well-known trace continuity Theorem
\cite{ad}
\[
y^{2}(0)\leq2\int_{0}^{1}\left(  y^{2}+y_{x}^{2}\right)  dx,
\]
the above inequality leads to the desired result with $A_{2}$ depending on
$m,M,|\alpha|,\tau,\beta$ and $||a||_{\infty}$.

With regard to the other inequality of (\ref{in}), we proceed as follows:
\begin{align}
& \Vert(y,z,u,\xi,\eta)\Vert_{\scriptscriptstyle{\mathcal{X}}}^{2}%
=\displaystyle\int_{0}^{1}\left(  ay_{x}^{2}+z^{2}\right)  dx+K\tau\int
_{0}^{1}u^{2}\,dx+m\xi^{2}+M\eta^{2}\nonumber\\
& +\varpi\left(  \int_{0}^{1}z\,dx+\tau\alpha\int_{0}^{1}u\,dx+m\xi
+M\eta\right)  ^{2}+\varpi\mu^{2}y^{2}(0)\nonumber\\
& +2\varpi\mu y(0)\left[  \int_{0}^{1}z\,dx+\tau\alpha\int_{0}%
^{1}u\,dx+m\xi+M\eta\right]  .\label{6}%
\end{align}
It follows from Young's inequality that for any $\kappa>0$,
\begin{gather}
2y(0)\left[  \int_{0}^{1}z\,dx+\tau\alpha\int_{0}^{1}u\,dx+m\xi+M\eta\right]
\geq\label{7}\\
-\dfrac{4}{\kappa}\left(  \left[  \int_{0}^{1}z\,dx\right]  ^{2}+\tau
^{2}\alpha^{2}\left[  \int_{0}^{1}u\,dx\right]  ^{2}+m^{2}\xi^{2}+M^{2}%
\eta^{2}\right)  -\kappa y^{2}(0).\nonumber
\end{gather}
Combining (\ref{6}) and (\ref{7}), and choosing $\kappa<\mu=\beta-\alpha$,
we obtain
\begin{align}
& \Vert(y,z,u,\xi,\eta)\Vert_{\scriptscriptstyle{\mathcal{X}}}^{2}%
\geq\nonumber\\
& \displaystyle\int_{0}^{1}ay_{x}^{2}dx+\left[  1+4\varpi\left(
1-\dfrac{\mu}{\kappa}\right)  \right]  \int_{0}^{1}z^{2}dx+\tau\left[
K+4\varpi\tau\alpha^{2}\left(  1-\dfrac{\mu}{\kappa}\right)  \right]
\int_{0}^{1}u^{2}\,dx\nonumber\\
& +m\left[  1+4m\varpi\left(  1-\dfrac{\mu}{\kappa}\right)  \right]
\xi^{2}+M\left[  1+4M\varpi\left(  1-\dfrac{\mu}{\kappa}\right)  \right]
\eta^{2}+\varpi\mu(\mu-\kappa)y^{2}(0).\label{8}%
\end{align}
A direct computation gives
\begin{align}
\displaystyle\int_{0}^{1}y^{2}dx  & =y^{2}(0)+2\displaystyle\int_{0}^{1}%
\int_{0}^{x}yy_{s}\;ds\;dx\nonumber\\
& \leq y^{2}(0)+{\varepsilon}\int_{0}^{1}y^{2}dx+\frac{1}{\varepsilon}\int
_{0}^{1}y_{x}^{2}dx,\label{n2n}%
\end{align}
for any $\varepsilon>0.$ Inserting (\ref{n2n}) into (\ref{8}) and using
(\ref{1.3}) yields
\begin{align}
& \Vert(y,z,u,\xi,\eta)\Vert_{\scriptscriptstyle{\mathcal{X}}}^{2}%
\geq\displaystyle\left[  a_{0}-\varepsilon^{-1}\varpi\mu(\mu
-\kappa)\right]  \int_{0}^{1}y_{x}^{2}dx+\varpi\mu(\mu-\kappa
)(1-\varepsilon)\int_{0}^{1}y^{2}dx\nonumber\\
& +\left[  1+4\varpi\left(  1-\dfrac{\mu}{\kappa}\right)  \right]  \int
_{0}^{1}z^{2}dx+\tau\left[  K+4\varpi\tau\alpha^{2}\left(  1-\dfrac{\mu
}{\kappa}\right)  \right]  \int_{0}^{1}u^{2}\,dx\nonumber\\
& +m\left[  1+4m\varpi\left(  1-\dfrac{\mu}{\kappa}\right)  \right]
\xi^{2}+M\left[  1+4M\varpi\left(  1-\dfrac{\mu}{\kappa}\right)  \right]
\eta^{2}+\varpi\mu(\mu-\kappa)y^{2}(0),\label{98}%
\end{align}
for any $0<\kappa<\mu=\beta-\alpha$ and $0<\varepsilon<1$. Finally, we
choose $\varpi$ such that
\[
0<\varpi<\min\left\{  \frac{\varepsilon a_{0}}{\mu(\mu-\kappa)}%
,\delta,\frac{\delta}{m},\frac{\delta}{M},\frac{K\delta}{\tau\alpha^{2}%
},\right\}  ,
\]
where $\delta=\frac{\kappa}{4(\mu-\kappa)}>0.$ Thus, (\ref{in}) holds and
the proof of Proposition \ref{p1} is achieved.
\end{proof}

We are now in a position to set our problem in the state space ${\mathcal{X}}
$. Define a linear operator $\mathcal{A}$ by
\begin{equation}%
\begin{array}
[c]{l}%
{\mathcal{D}}(\mathcal{A})=\left\{  (y,z,u,\xi,\eta)\in{\mathcal{X}};y\in
H^{2}(0,1),\;\;z,u\in H^{1}(0,1),\;\;\xi=u(0)=z(0),\;\eta=z(1)\right\}  ,\\
\displaystyle{\mathcal{A}}(y,z,u,\xi,\eta)=\left(  z,(ay_{x})_{x},-\frac
{u_{x}}{\tau},\frac{1}{m}\left[  (ay_{x})(0)-\beta\xi+\alpha u(0)\right]
,-\dfrac{(ay_{x})(1)}{M}\right)  ,\\
\forall(y,z,u,\xi,\eta)\in{\mathcal{D}}(\mathcal{A}).\label{1.62}%
\end{array}
\end{equation}
The closed-loop system (\ref{1.6}) can now be formulated in terms of the
operator $\mathcal{A}$\ by the evolution equation over ${\mathcal{X}}$
\begin{equation}
\left\{
\begin{array}
[c]{l}%
\dot{\Phi}(t)=\mathcal{A}\Phi(t),\\
\Phi(0)=\Phi_{0},
\end{array}
\right. \label{si}%
\end{equation}
in which $\Phi=(y,z,u,\xi,\eta)$ and $\Phi_{0}=(y_{0},y_{1},f(-\tau\cdot
),\xi_{0},\eta_{0}).$

\noindent The well-posedness result is stated below.

\begin{theo}
Suppose that (\ref{1.3}), (\ref{sma}) and (\ref{uni}) are satisfied. Then, we have:

\noindent(i) The operator $\mathcal{A}$ defined by (\ref{1.62}) is densely
defined in ${\mathcal{X}}$ and generates on ${\mathcal{X}}$ a $C_{0}%
$-semigroup of contractions $e^{t\mathcal{A}}$. Moreover, $\sigma
(\mathcal{A})$, the spectrum of $\mathcal{A}$, consists of isolated
eigenvalues of finite algebraic multiplicity only.

\noindent(ii) For any initial condition $\Phi_{0}\in{\mathcal{X}}$, the system
(\ref{si}) has a unique mild solution $\Phi\in C([0,\infty);\mathcal{X})$. In
turn, if $\Phi_{0}\in{\mathcal{D}}(\mathcal{A})$, then necessarily the
solution $\Phi$ is strong and belongs to $C([0,\infty);{\mathcal{D}%
}(\mathcal{A})\cap C^{1}([0,\infty);\mathcal{X})$. \label{(t1)}
\end{theo}

\begin{proof}
Let $\Phi=(y,z,u,\xi,\eta)\in{\mathcal{D}}({\mathcal{A}}).$ Then, in light of
(\ref{ip}) and (\ref{1.62}), a simple integration by parts gives
\[
\langle{\mathcal{A}}\Phi,\Phi)\rangle_{{\mathcal{X}}}=\displaystyle(ay_{x}%
)(1)z(1)-(ay_{x})(0)z(0)-\dfrac{K}{2}(u^{2}(1)-u^{2}(0))+\xi(ay_{x}%
)(0)-\beta\xi^{2}+\alpha\xi u(1)
\]%
\[
-\eta(ay_{x})(1)+\varpi\left(  \int_{0}^{1}z\,dx+\tau\alpha\int_{0}%
^{1}u\,dx+m\xi+M\eta\right)  \left(  \alpha u(0)+\beta\xi+(\beta
-\alpha)z(0)\right)
\]%
\[
=\alpha\xi u(1)-\displaystyle\frac{K}{2}u^{2}(1)+\dfrac{K}{2}u^{2}(0)-\beta
\xi^{2}%
\]%
\begin{equation}
\leq\left(  -\beta+\dfrac{K+|\alpha|}{2}\right)  \xi^{2}+\displaystyle\frac
{|\alpha|-K}{2}u^{2}(1)\label{dis1}%
\end{equation}
and so the operator ${\mathcal{A}}$ is dissipative due to the assumption
(\ref{uni}).

Next, we claim that the operator $\lambda I-\mathcal{A}$ is onto $\mathcal{X}
$ for $\lambda>0$ sufficiently large. To ascertain the correctness of this
claim, one has to show that given $(f,g,v,p,q)\in{\mathcal{X}}$, there exists
$(y,z,u,\xi,\eta)\in{\mathcal{D}}(\mathcal{A})$ for which $(\lambda
I-\mathcal{A})(y,z,u,\xi,\eta)=(f,g,v,p,q)$. Although this can be considered
as a classical problem, one can easily verify that the latter is equivalent to
solve the following system
\begin{equation}
\left\{
\begin{array}
[c]{l}%
\lambda^{2}y-(ay_{x})_{x}=\lambda f+g,\\
u_{x}+\lambda\tau u=\tau v,\\
\lambda(m\lambda+\beta)y(0)-(ay_{x})(0)-\alpha u(0)=mp+(m\lambda+\beta)f(0),\\
\lambda^{2}My(1)+(ay_{x})(1)=Mq+\lambda Mf(1),\\
z=\lambda y-f,\\
\xi=u(0)=z(0)=\lambda y(0)-f(0),\\
\eta=z(1)=\lambda y(1)-f(1).
\end{array}
\right. \label{ra}%
\end{equation}
Solving the equation of $u$ in the above system, we obtain
\begin{equation}%
\begin{array}
[c]{l}%
\displaystyle u(x)=e^{-{\tau}\lambda x}(\lambda y(0)-f(0))+\displaystyle{\tau
}\int_{0}^{x}e^{-{\tau}\lambda(x-s)}v(s)\,ds,
\end{array}
\label{uu1}%
\end{equation}
and hence
\begin{equation}
\displaystyle u(1)=e^{-{\tau}\lambda}(\lambda y(0)-f(0))+\displaystyle{\tau
}\int_{0}^{1}e^{-{\tau}\lambda(1-s)}v(s)\,ds.\label{uu2}%
\end{equation}
This, together with (\ref{ra}) and (\ref{uu1}), imply that one has only to
seek $y\in H^{2}(0,1)$ satisfying
\begin{equation}
\left\{
\begin{array}
[c]{l}%
\lambda^{2}y-(ay_{x})_{x}=\lambda f+g,\\
\lambda\left[  (m\lambda+\beta)-\alpha e^{-{\tau}\lambda}\right]
y(0)-(ay_{x})(0)=mp+(m\lambda+\beta-\alpha e^{-{\tau}\lambda})f(0)\\
\hspace{6.9cm}+\displaystyle{\tau}\alpha\int_{0}^{1}e^{-{\tau}\lambda
(1-s)}v(s)\,ds,\\
\lambda^{2}My(1)+(ay_{x})(1)=Mq+\lambda Mf(1).
\end{array}
\right. \label{ra1}%
\end{equation}
Multiplying the first equation in (\ref{ra1}) by $\phi\in H^{1}(0,1)$, we get
the weak formulation
\[
\displaystyle\int_{0}^{1}\left(  \lambda^{2}\phi y+ay_{x}\phi_{x}\right)
dx+\lambda\left[  (m\lambda+\beta)-\alpha e^{-{\tau}\lambda}\right]
y(0)\phi(0)+\lambda^{2}My(1)\phi(1)
\]%
\[
=\int_{0}^{1}\left(  \lambda f+g\right)  \phi\,dx+\left[  mp+(m\lambda
+\beta-\alpha e^{-{\tau}\lambda})f(0)+\displaystyle{\tau}\alpha\int_{0}%
^{1}e^{-{\tau}\lambda(1-s)}v(s)\,ds\right]  \phi(0)
\]%
\begin{equation}
+\left[  Mq+\lambda Mf(1)\right]  \phi(1),\label{(1I)}%
\end{equation}
which in turn can be written in the form ${\mathcal{L}}(y,\phi)={\mathcal{M}%
}(\phi),$ where ${\mathcal{L}}$ is a bilinear form defined by
\[
{\mathcal{L}}:H^{1}(0,1)\times H^{1}(0,1)\longrightarrow\mathbb{R}%
\]
such that
\[
{\mathcal{L}}(y,\phi)=\displaystyle\int_{0}^{1}\left(  \lambda^{2}\phi
y+ay_{x}\phi_{x}\right)  dx+\lambda\left[  (m\lambda+\beta)-\alpha e^{-{\tau
}\lambda}\right]  y(0)\phi(0)+\lambda^{2}My(1)\phi(1),
\]
and ${\mathcal{M}}$ is a linear form given by
\[
{\mathcal{M}}:H^{1}(0,1)\longrightarrow\mathbb{R}%
\]%
\[
\displaystyle\phi\longmapsto{\mathcal{M}}(\phi)=\int_{0}^{1}\left(  \lambda
f+g\right)  \phi\,dx+
\]%
\[
\left[  mp+(m\lambda+\beta-\alpha e^{-{\tau}\lambda})f(0)+\displaystyle{\tau
}\alpha\int_{0}^{1}e^{-{\tau}\lambda(1-s)}v(s)\,ds\right]  \phi(0)
\]%
\[
+\left[  Mq+\lambda Mf(1)\right]  \phi(1).
\]
Applying Lax-Milgram Theorem \cite{br}, one can deduce the existence of a
unique solution $y\in H^{2}(0,1)$ of (\ref{ra1}) as long as $\lambda>0$ is
large. This establishes that the range of $\lambda I-\mathcal{A}$ is
${\mathcal{X}}$, for $\lambda>0$. Thus, according to semigroup theory
\cite{Pa:83}, the operator $\mathcal{A}$ is densely defined in ${\mathcal{X}}$
and generates on ${\mathcal{X}}$ a $C_{0}$-semigroup of contractions denoted
by $e^{t\mathcal{A}}$. As a direct consequence of the fact that, for
$\lambda>0$, the range of $\lambda I-\mathcal{A}$ is ${\mathcal{X}}$, it
follow that $\left(  \lambda I-{\mathcal{A}}\right)  ^{-1}$ exists and maps
${\mathcal{X}}$ into ${\mathcal{D}}({\mathcal{A}})$. Finally, using Sobolev
embedding \cite{ad}, if follows that $\left(  \lambda I-{\mathcal{A}}\right)
^{-1}$ is compact and hence the spectrum of $\mathcal{A}$, consists of
isolated eigenvalues of finite algebraic multiplicity only \cite{ka}. This
completes the proof of the first assertion (i) in Theorem \ref{(t1)}.

Concerning the proof of the second assertion, it suffices to use (i) and
invoke semigroups theory \cite{Pa:83}.
\end{proof}

\section{Asymptotic behavior}
\label{sect3}

We begin this section by recalling the following result.

\begin{theo}
\cite{HA} \label{sta} Let $P$ be the infinitesimal generator of a $C_{0}%
$-semigroup $S(t)$ in a Hilbert space $H$ such that $P$ has compact resolvent.
Then, $S(t)$ is strongly stable if and only if it is uniformly bounded and
$\Re\lambda<0$, for any $\lambda$ in the spectrum of $P$.
\end{theo}

It is clear from (\ref{1.62}) that $\lambda=0$ is an eigenvalue of
$\mathcal{A}$ whose eigenfunction is $(c,0,0,0,0)$, where $c\in\mathbb{R}%
\setminus\{0\}$. Thus, Theorem \ref{sta} implies that the semigroup
$e^{t\mathcal{A}}$ generated by $\mathcal{A}$ is not stable. However, we are
able to prove the main result of the paper which is stated next.

\begin{theo}
Assume that (\ref{1.3}), (\ref{sma}) holds and $K$ satisfies $|\alpha
|<K<2\beta-|\alpha|$. Then, for any initial data $\Phi_{0}=(y_{0},y_{1}%
,f,\xi_{0},\eta_{0})\in{\mathcal{X}}$, the solution $\Phi(t)=\biggl(y,y_{t}%
,y_{t}(0,t-x\tau),y_{t}(0,t),y_{t}(1,t)\biggr)$ of the closed-loop system
(\ref{1.6}) (or equivalently (\ref{si})) tends in ${\mathcal{X}}$ to
$(\Omega,0,0,0,0)$ as $t\longrightarrow+\infty$, where
\begin{equation}
\Omega=\displaystyle\frac{1}{\beta-\alpha}\displaystyle\left[  \int_{0}%
^{1}y_{1}dx+\alpha\tau\int_{0}^{1}f(-\tau x)dx+(\beta-\alpha)y_{0}(0)+m\xi
_{0}+M\eta_{0}\right]  .\label{cste}%
\end{equation}
\label{t2}
\end{theo}

\begin{proof}
The proof depends on an essential way on the application of LaSalle's
invariance principle \cite{HA}. Using a standard argument of density of
${\mathcal{D}}({\mathcal{A}})$ in $\mathcal{X}$ and the contraction of the
semigroup $e^{t\mathcal{A}}$, it suffices to prove Theorem \ref{t2} for smooth
initial data $\Phi_{0}=\left(  y_{0},y_{1},f,\xi_{0},\eta_{0}\right)
\in{\mathcal{D}}({\mathcal{A}})$. Let $\Phi(t)=\left(  y(t),y_{t}%
(t),u(t),\xi(t),\eta(t)\right)  =e^{t\mathcal{A}}\Phi_{0}$ be the solution of
(\ref{1.6}). It follows from Theorem \ref{(t1)} that the trajectories set of
solutions $\{\Phi(t)\}_{\scriptscriptstyle t\geq0}$ is a bounded for the graph
norm and thus precompact by virtue of the compactness of the operator $\left(
I-\mathcal{A}\right)  ^{-1}$. Invoking LaSalle's principle, we deduce that
$\omega\left(  \Phi_{0}\right)  $ is non empty, compact, invariant under the
semigroup $e^{t\mathcal{A}}$ and in addition $e^{t\mathcal{A}}\Phi
_{0}\longrightarrow\omega\left(  \Phi_{0}\right)  \;$ as $t\rightarrow
\infty\,$ \cite{HA}. Clearly, in order to prove the convergence result, it
suffices to show that $\omega\left(  \Phi_{0}\right)  $ reduces to $({\Omega
},0,0,0)$. To this end, let $\tilde{\Phi}_{0}=\left(  \tilde{y}_{0},\tilde
{y}_{1},\tilde{f},\tilde{\xi},\tilde{\eta}\right)  \in\omega\left(  \Phi
_{0}\right)  \subset{D}(\mathcal{A})$ and consider $\tilde{\Phi}(t)=\left(
\tilde{y}(t),\tilde{y}_{t}(t),\tilde{u}(t),\tilde{\xi}(t),\tilde{\eta
}(t)\right)  =e^{t\mathcal{A}}\tilde{\Phi}_{0}\in{D}(\mathcal{A})$ as the
unique strong solution of (\ref{si}). It is well-known that $\Vert\tilde{\Phi
}(t)\Vert_{\mathcal{X}}$ is constant \cite{HA} and thus $\frac{{\textstyle d}%
}{\textstyle dt}\left(  \Vert\tilde{\Phi}(t)\Vert_{\mathcal{X}}^{2}\right)
=0$. This leads to
\begin{equation}
<\mathcal{A}\tilde{\Phi},\tilde{\Phi}>_{\mathcal{X}}=0\label{e1n}%
\end{equation}
which, together with (\ref{dis1}), imply that $\tilde{\xi}=\tilde{y}%
_{t}(0,t)=0$ and $\tilde{u}(1)=\tilde{y}_{t}(0,t-\tau)=0$. Consequently,
$\tilde{y}$ is a solution of the system
\begin{equation}
\left\{
\begin{array}
[c]{ll}%
\tilde{y}_{tt}-(a\tilde{y}_{x})_{x}=0, & (x,t)\in(0,1)\times(0,\infty),\\
M\tilde{y}_{tt}(1,t)+(a\tilde{y}_{x})(1,t)=0, & t>0,\\
\tilde{y}_{t}(0,t)=\tilde{y}_{x}(0,t)=0, & t>0,\\
\tilde{y}(0)=\tilde{y}_{0};\,\tilde{y}_{t}(0)=\tilde{y}_{1}, & x\in(0,1)\\
\tilde{y}\in H^{2}(0,1). &
\end{array}
\right. \label{e2n}%
\end{equation}
A straightforward computation shows that $\tilde{z}=\tilde{y}_{t}$ is a
solution of
\begin{equation}
\left\{
\begin{array}
[c]{ll}%
\tilde{z}_{tt}-(a\tilde{z}_{x})_{x}=0, & (x,t)\in(0,1)\times(0,\infty),\\
M\tilde{z}_{tt}(1,t)+(a\tilde{z}_{x})(1,t)=0, & t>0,\\
\tilde{z}(0,t)=\tilde{z}_{x}(0,t)=0, & t>0,\\
\tilde{z}(0)=\tilde{y}_{1};\,\tilde{z}_{t}(0)=(a\tilde{y_{0}}_{x})_{x}, &
x\in(0,1).
\end{array}
\right. \label{e3n}%
\end{equation}
The problem (\ref{e3n}) admits only the trivial solution $\tilde{z}=0$. The
arguments used to prove this run on much the same lines as in \cite{Ra} (see
also \cite{mif2}). Consequently, the unique solution of (\ref{e2n}),
$\tilde{y}$, is constant. To summarize, we have shown that for any
$\tilde{\Phi}_{0}=\left(  \tilde{y}_{0},\tilde{y}_{1},\tilde{f},\tilde{\xi
},\tilde{\eta}\right)  \in\omega\left(  \Phi_{0}\right)  \subset
{D}({\mathcal{A}})$, the unique solution $\tilde{\Phi}(t)=\left(  \tilde
{y}(t),\tilde{y}_{t}(t),\tilde{u}(t),\tilde{\xi}(t),\tilde{\eta}(t)\right)
=e^{t\mathcal{A}}\tilde{\Phi}_{0}\in{D}(\mathcal{A})$ is actually
$(\Omega,0,0,0,0)$, for any $t\geq0$, where $\Omega$ is a constant to be
determined. This implies that the initial condition $\tilde{\Phi}_{0}=\left(
\tilde{y}_{0},\tilde{y}_{1},\tilde{f},\tilde{\xi},\tilde{\eta}\right)  $ is
also equal to $(\Omega,0,0,0,0)$. Thereby, the $\omega$-limit set
$\omega\left(  \Phi_{0}\right)  $ only consists of constants $(\Omega
,0,0,0,0)$. It remains to provide an explicit expression of the constant
$\Omega$ to complete the proof. To do so, let $(\Omega,0,0,0,0)\in
\omega\left(  \Phi_{0}\right)  $. This implies that there exists $\left\{
t_{n}\right\}  \rightarrow\infty$, as $n\rightarrow\infty$ such that
\begin{equation}
\Phi(t_{n})=(y(t_{n}),y_{t}(t_{n}),y_{t}(t_{n}),u(t_{n}),\xi(t_{n}),\eta
(t_{n}))=e^{t_{n}\mathcal{A}}\Phi_{0}\longrightarrow(\Omega
,0,0,0,0)\label{boum}%
\end{equation}
in the state space $\mathcal{X}$. Furthermore, in view to Remark \ref{zer},
any solution of the closed-loop system (\ref{si}) stemmed from $\Phi
_{0}=(y_{0},y_{1},f,\xi_{0},\eta_{0})$ verifies
\[
\int_{0}^{1}y_{t}(x,t)\,dx+my_{t}(0,t)+My_{t}(1,t)+\alpha\tau\int_{0}^{1}%
y_{t}(0,t-x\tau)\,dx+(\beta-\alpha)y(0,t)=\Upsilon,\;\forall t\geq0,
\]
in which $\Upsilon$ is a constant. Obviously, such a constant can be obtained
by taking $t=0$ in the left-hand side of the last equation. Therefore, we
have
\begin{align}
& \int_{0}^{1}y_{t}(x,t)\,dx+my_{t}(0,t)+My_{t}(1,t)+\alpha\tau\int_{0}%
^{1}y_{t}(0,t-x\tau)\,dx+(\beta-\alpha)y(0,t)\nonumber\\
& =\int_{0}^{1}y_{t}(x,0)\,dx+my_{t}(0,0)+My_{t}(1,0)+\alpha\tau\int_{0}%
^{1}y_{t}(0,-x\tau)\,dx+(\beta-\alpha)y(0,0)\nonumber\\
& =\int_{0}^{1}y_{1}(x)\,dx+m\xi_{0}+M\eta_{0}+\alpha\tau\int_{0}^{1}%
f(-x\tau)\,dx+(\beta-\alpha)y_{0}(0).\label{bou}%
\end{align}
Lastly, letting $t=t_{n}$ in (\ref{bou}) with $n\rightarrow\infty$ and using
(\ref{boum}) yield the desired expression of $\Omega$. This achieves the proof
of the theorem.
\end{proof}

\section{Polynomial convergence}

\label{sect4}

The objective of this section is to show that the convergence result obtained
in the previous section is in fact polynomial. The proof of such a desired
result is based on applying the following frequency domain theorem for
polynomial stability of a $C_{0}$ semigroup of contractions on a Hilbert space
\cite{tomilov}:

\begin{theo}
\label{lemraokv} A $C_{0}$ semigroup $e^{t{\mathcal{L}}}$ of contractions on a
Hilbert space ${\mathcal{H}}$ satisfies, for all $t>0$,
\[
||e^{t{\mathcal{L}}}||_{{\mathcal{L}}(\mathcal{D}(\mathcal{A}),{\mathcal{H}}%
)}\leq\frac{C}{t^{1/\delta}}%
\]
for some constant $C,\delta>0$ if and only if
\begin{equation}
\rho({\mathcal{L}})\supset\bigr\{i\gamma\bigm|\gamma\in\mathbb{R}\bigr\}\equiv
i\mathbb{R},\label{1.8wkv}%
\end{equation}
and
\begin{equation}
\limsup_{|\gamma|\rightarrow\infty}\Vert|\gamma|^{\delta}\,(i\gamma
I-{\mathcal{L}})^{-1}\Vert_{{\mathcal{L}}({\mathcal{H}})}<\infty,\label{1.9kv}%
\end{equation}
where $\rho({\mathcal{L}})$ denotes the resolvent set of the operator
${\mathcal{L}}$.
\end{theo}

In order to use the above theorem, let us first consider the space
\[
\dot{\mathcal{X}} = \left\{ (y,z,u,\xi,\eta) \in\mathcal{X}; \, \int_{0}^{1}
z(x) dx +\alpha\tau\int_{0}^{1} u(x) dx + (\beta-\alpha) \, y(0) + m \xi+ M
\eta= 0 \right\} .
\]
Then, a new operator is defined below
\[
\dot{\mathcal{A}} : \mathcal{D}(\dot{\mathcal{A}}) := \mathcal{D}(\mathcal{A})
\cap\dot{\mathcal{X}} \subset\dot{\mathcal{X}} \rightarrow\dot{\mathcal{X}},
\]
\begin{equation}
\label{1.62bis}\dot{\mathcal{A}} (y,z,u,\xi,\eta) = \mathcal{A} (y,z,u,\xi
,\eta), \, \forall\, (y,z,u,\xi,\eta) \in\mathcal{D}(\dot{\mathcal{A}}).
\end{equation}
Clearly, the operator $\dot{\mathcal{A}}$ defined by (\ref{1.62bis}) generates
on $\dot{{\mathcal{X}}}$ a $C_{0}$-semigroup of contractions $e^{t
\dot{\mathcal{A}}}$ provided that the conditions (\ref{sma}) and (\ref{uni})
are fulfilled. Moreover, $\sigma(\dot{\mathcal{A}})$, the spectrum of
$\dot{\mathcal{A}}$, consists of isolated eigenvalues of finite algebraic
multiplicity only. In order to achieve the objective of this section, we shall
assume that the coefficient $a$ satisfies stronger conditions than
(\ref{1.3}), namely,
\begin{equation}
\left\lbrace
\begin{array}
[c]{l}%
a \in C^{1} [0,1 ];\\
\text{there exist positive constants} \; a_{0}, a_{1} \; \text{such that} \;
a(x) \ge a_{0} , \; a^{\prime}(x) \geq a_{1}, \; \text{for all} \;\; x
\in[0,\ell].
\end{array}
\right. \label{1.33}%
\end{equation}

Now, we are ready to state our result which translates the fact that the
semigroup operator $e^{t \dot{\mathcal{A}}}$ is polynomially stable in
$\dot{\mathcal{X}}$.

\begin{theo}
\label{lrkv}Assume that (\ref{sma}) and (\ref{1.33}) hold and $K$ satisfies
$|\alpha|<K<2\beta-|\alpha|$. Then, there exists $C>0$ such that for all $t>0$
we have
\[
\left\Vert e^{t\dot{\mathcal{A}}}\right\Vert _{{\mathcal{L}}(\mathcal{D}%
(\dot{\mathcal{A}}),\dot{\mathcal{X}})}\leq\frac{C}{\sqrt{t}}.
\]

\end{theo}

\begin{proof}
[Proof of theorem \ref{lrkv}]The proof of Theorem \ref{lrkv} is based on the
following lemmas.

We first look at the point spectrum.
\end{proof}

\begin{lem}
\label{condspkv} If $\gamma$ is a real number, then $i\gamma$ is not an
eigenvalue of $\dot{\mathcal{A}}$.
\end{lem}

\begin{proof}
We will show that the equation
\begin{equation}
\dot{\mathcal{A}}Z=i \gamma Z\label{eigenkv}%
\end{equation}
with $Z=(y,z,u,\xi,\eta)\in\mathcal{D}(\dot{\mathcal{A}})$ and $\gamma
\in\mathbb{R}$ has only the trivial solution. Clearly, the system
\eqref{eigenkv} writes
\begin{align}
z  & =i\gamma y\label{eigen1}\\
(a\,y_{x})_{x}  & =i\gamma z\label{eigen2}\\
-\frac{u_{x}}{\tau}  & =i\gamma u\label{eigen3}\\
\frac{1}{m}\left[  (ay_{x})(0)-\beta\xi+\alpha u(0)\right]   & =i\gamma
\xi\ .\label{eigen4}\\
-\frac{(ay_{x})(1)}{M}  & =i\gamma\eta.\label{eigenbis}%
\end{align}
\noindent Let us firstly treat the case where $\gamma=0$. It's clear that the
only solution of \eqref{eigenkv} is the trivial one.

\medskip

Suppose now that $\gamma\neq0$. By taking the inner product of (\ref{eigenkv})
with $Z$, using the inequality (\ref{dis1}) we get:
\begin{equation}
\Re\left(  <\dot{\mathcal{A}}Z,Z>_{{\dot{\mathcal{X}}}}\right)  \leq\frac
{1}{2}\,\left(  (-2\beta+|\alpha|+K)|z(0)|^{2}+(|\alpha|-K)|u(1)|^{2}\right)
(\leq0).\label{1.7kv}%
\end{equation}
Thenceforth, we obtain that $z(0)=0$ and $u(1)=0$ and hence $\xi=u(0)=0$. Lastly, we
conclude that the only solution of \eqref{eigenkv} is the trivial one.
\end{proof}

\begin{lem}
\label{lemresolventkv} The resolvent operator of $\dot{\mathcal{A}}$ obeys the
condition \eqref{1.9kv}.
\end{lem}

\begin{proof}
Suppose that condition \eqref{1.9kv} is false. By the Banach-Steinhaus Theorem
(see \cite{br}), there exist a sequence of real numbers $\gamma_{n}%
\rightarrow+\infty$ and a sequence of vectors \newline$Z_{n}=(y_{n}%
,z_{n},u_{n},\xi_{n},\eta_{n})\in\mathcal{D}(\dot{\mathcal{A}})$ with $\Vert
Z_{n}\Vert_{\dot{\mathcal{X}}}=1$ such that
\begin{equation}
\Vert\gamma_{n}^{2}\,(i\gamma_{n}I-\dot{\mathcal{A}})Z_{n}\Vert_{\dot
{\mathcal{X}}}\rightarrow0\;\;\;\;\mbox{as}\;\;\;n\rightarrow\infty
,\label{1.12kv}%
\end{equation}
that is, {as} $n\rightarrow\infty$, we have:
\begin{equation}
\gamma_{n}^{2}\left(  i\gamma_{n}y_{n}-z_{n}\right)  \equiv\gamma_{n}%
^{2}\,f_{n}\rightarrow0\;\;\;\mbox{in}\;\;H^{1}(0,1),\label{1.13kv}%
\end{equation}%
\begin{equation}
\gamma_{n}^{2}\left(  i\gamma_{n}z_{n}-(a(y_{n})_{x})_{x}\right)  \equiv
\gamma_{n}^{2}\,g_{n}\rightarrow0\;\;\;\mbox{in}\;\;L^{2}(0,1),\label{1.13bkv}%
\end{equation}%
\begin{equation}
\gamma_{n}^{2}\left(  i\gamma_{n}u_{n}+\frac{(u_{n})_{x}}{\tau}\right)
\equiv\gamma_{n}^{2}\,v_{n}\rightarrow0\;\;\;\mbox{in}\;\;L^{2}%
(0,1),\label{1.14bkv}%
\end{equation}%
\begin{equation}
\gamma_{n}^{2}\left(  i\gamma_{n}\xi_{n}-\frac{1}{m}\left[  (a(y_{n}%
)_{x})(0)-\beta\xi_{n}+\alpha u_{n}(1)\right]  \right)  \equiv\gamma_{n}%
^{2}\,p_{n}\rightarrow0,\label{1.14kv}%
\end{equation}%
\begin{equation}
\gamma_{n}^{2}\left(  i\gamma_{n}\eta_{n}+\frac{(a(y_{n})_{x})(1)}{M}\right)
\equiv\gamma_{n}^{2}\,q_{n}\rightarrow0.\label{lasteq}%
\end{equation}

Our goal is to derive from \eqref{1.12kv} that $\Vert Z_{n}\Vert
_{\dot{\mathcal{X}}}$ converges to zero, thus there is a contradiction. The
proof is divided into three steps\medskip

\noindent\textbf{First step.}\medskip

We first notice that we have
\begin{equation}
||\gamma_{n}^{2}(i\gamma_{n}I-\dot{\mathcal{A}})Z_{n}||_{\dot{\mathcal{X}}%
}\geq\left\vert \Re\left(  \langle\gamma_{n}^{2}\,(i\gamma_{n}I-\dot
{\mathcal{A}})Z_{n},Z_{n}\rangle_{\dot{\mathcal{X}}}\right)  \right\vert
\ .\label{1.15kv}%
\end{equation}
Amalgamating  \eqref{1.15kv} with \eqref{1.7kv}-\eqref{1.13kv}, it follows that
\[
\gamma_{n}\,z_{n}(0)\rightarrow0,
\]
and
\[
\gamma_{n}\,u_{n}(1)\rightarrow0.
\]
Moreover, since $Z_{n}\in{\mathcal{D}}(\dot{\mathcal{A}})$, we deduce that
$\xi_{n}=\,u_{n}(0)=z_{n}(0)$. Thereby
\begin{equation}
\gamma_{n}\,\xi_{n}=\gamma_{n}\,u_{n}(0)\rightarrow0\;\hbox{and}\;\xi
_{n}\rightarrow0.\label{unkv}%
\end{equation}
Whereupon, (\ref{1.14kv}) gives
\begin{equation}
(y_{n})_{x}(0)\rightarrow0.\label{grad0}%
\end{equation}
Solving (\ref{1.14bkv}), we have the following identity
\[
u_{n}(x)=u_{n}(0)\,e^{-i\tau\gamma_{n}x}+\tau\,\int_{0}^{x}e^{-i\tau\gamma
_{n}(x-s)}v_{n}(s)\,ds\ ,
\]
which, together with (\ref{unkv}), implies that
\begin{equation}
u_{n}\rightarrow0\;\;\mbox{in}\;\;L^{2}(0,1).\label{znkv}%
\end{equation}
We have according to (\ref{1.12kv})
\begin{equation}
y_{n}=\frac{1}{i\gamma_{n}}(z_{n}+f_{n})\rightarrow0\;\hbox{in}\;L^{2}%
(0,1).\label{eqcle0}%
\end{equation}
Invoking \eqref{1.13kv}, \eqref{1.14kv} and \eqref{lasteq}, we have
\[
\int_{0}^{1}z_{n}(x)dx-\frac{1}{i\gamma_{n}}\int_{0}^{1}g_{n}(x)dx=\frac
{p_{n}}{i\gamma_{n}}-m\xi_{n}-M\eta_{n}+\frac{q_{n}}{i\gamma_{n}}%
+\circ(1)=-M\eta_{n}+\circ(1).
\]
Then, since $Z_{n}\in\dot{\mathcal{X}}$, we obtain that
\begin{equation}
y_{n}(0)\rightarrow0.\label{z1kv}%
\end{equation}
Integrating (\ref{1.13bkv}) we get
\[
\int_{0}^{1}i\gamma_{n}\,z_{n}(x)\,a\,\overline{(y_{n})_{x}}(x)\,dx-\int
_{0}^{1}(a(y_{n})_{x})_{x}(x)a(x)\overline{(y_{n})_{x}}(x)\,dx=\int_{0}%
^{1}g_{n}(x)a(x)\,\overline{(y_{n})_{x}}(x)\,dx
\]
and hence
\[
\frac{1}{\gamma_{n}^{2}}\left(  \left\vert (a(1)(y_{n})_{x})(1)\right\vert
^{2}-\left\vert (a(0)(y_{n})_{x})(0)\right\vert ^{2}\right)  =\circ(1).
\]
Therefore, using (\ref{1.14kv}) and (\ref{lasteq}), we have
\begin{equation}
M^{2}|\eta_{n}|^{2}-m^{2}|\xi_{n}|^{2}=\circ(1)\Rightarrow\eta_{n}%
\rightarrow0.\label{cveta}%
\end{equation}
Since $(f_{n},g_{n},v_{n},p_{n},q_{n})\in\dot{\mathcal{X}}$, we obtain
\[
\int_{0}^{1}g_{n}(x)\,dx+\alpha\tau\,\int_{0}^{1}v_{n}(x)\,dx+(\beta
-\alpha)\,f_{n}(0)+m\,\xi_{n}+M\,\eta_{n}=0,\forall\,n\in\mathbb{N}.
\]
Therefore
\begin{equation}
f_{n}(0)\rightarrow0\;\hbox{and}\;f_{n}(1)=\int_{0}^{1}(f_{n})_{x}%
(x)\,dx+f_{n}(0)\rightarrow0.\label{cvf}%
\end{equation}
This also implies, thanks to (\ref{1.13kv}), that
\[
i\gamma_{n}y_{n}(1)=z_{n}(1)+f_{n}(1)\rightarrow0,\,i\gamma_{n}y_{n}%
(0)=z_{n}(0)+f_{n}(0)\rightarrow0.
\]

\noindent\textbf{Second step.}\medskip

We express now $z_{n}$ in terms of $y_{n}$ from equation (\ref{1.13kv}) and
substitute it into (\ref{1.13bkv}) to get
\begin{equation}
-\gamma_{n}^{2}y_{n}-(a(y_{n})_{x})_{x}=i\gamma_{n}f_{n}+g_{n}.\label{cvzy}%
\end{equation}
Next, we take the inner product of (\ref{cvzy}) with $b(y_{n})_{x}$ in
$L^{2}(0,1)$, where $b\in C^{1}([0,1])$. We obtain
\begin{gather}
\int_{0}^{1}\left(  -\gamma_{n}^{2} y_{n}(x)-(a(y_{n})_{x})_{x}(x)\right)
\,b(x)\overline{(y_{n})_{x}}(x)\,dx=\int_{0}^{1}\left(  i\gamma_{n}%
f_{n}(x)+g_{n}(x)\right)  \,b(x)\overline{(y_{n})_{x}}(x)\,dx = \nonumber\\
-\int_{0}^{1}i\gamma_{n}(f_{n})_{x}(x)b(x)\overline{y}_{n}(x)\,dx+\int
_{0}^{1}g_{n}(x)b(x)(\overline{y_{n}})_{x}(x)\,dx+i\gamma_{n}f_{n}%
(1)b(1)\overline{y}_{n}(1)-i\gamma_{n}f_{n}(0)b(0)\overline{y}_{n}%
(0).\label{step1}%
\end{gather}
It is clear that the right-hand side of (\ref{step1}) converges to zero since
$f_{n},\,g_{n},\,f_{n}(0),\,f_{n}(1),$ $\gamma_{n}y_{n}(0)$ and $\gamma
_{n}y_{n}(1)$ converge to zero in $H^{1}(0,1),L^{2}(0,1)$ and $\mathbb{C}$,
respectively.\medskip

On the other hand, a straightforward calculation yields
\[
\Re\left\{  \int_{0}^{1}-\gamma_{n}^{2}\,y_{n}b(x)\overline{(y_{n})}
_{x}(x)\,dx\right\}  =\frac{1}{2}\,\int_{0}^{1}\gamma_{n}^{2}\,b^{\prime
}(x)|y_n(x)|^{2}\,dx - \frac{1}{2}\,\left( -\gamma_{n}^{2}b(1)|y_{n}
(1)|^{2}+\gamma_{n}^{2}b(0)\,|y_n(0)|^{2}\right)
\]
\[
=\frac{1}{2} \, \int_{0}^{1}\gamma_{n}^{2}\,b^{\prime}(x)|y_{n}(x)|^{2}%
\,dx + \circ(1)
\]
and
\[
\Re\left\{  -\int_{0}^{1}(a(y_{n})_{x})_{x}(x)\,b(x)\overline{(y_{n})}%
_{x}(x)\,dx\right\}  =\frac{1}{2}\,\int_{0}^{1}(ab^{\prime}-a^{\prime
}b)|(y_{n})_{x}|^{2}\,dx
\]%
\[
-\frac{1}{2}\left(  a(1)b(1)\,|((y_{n})_{x}(1)|^{2}-a(0)b(0)\,|(y_{n}%
)_{x}(0)|^{2}\right)
\]%
\[
=\frac{1}{2}\,\int_{0}^{1}(ab^{\prime}-a^{\prime}b)|(y_{n})_{x}|^{2}%
\,dx-\frac{1}{2}\,a(1)b(1)\,|((y_{n})_{x}(1)|^{2}+\circ(1).
\]
This leads to
\begin{equation}
\int_{0}^{1}(ab^{\prime}-a^{\prime}b)|(y_{n})_{x}|^{2}\,dx+\int_{0}%
^{1} \gamma_{n}^{2} b^{\prime}(x)\,|y_{n}(x)|^{2}\,dx-a(1)b(1)|(y_{n})_{x}(1)|^{2}%
=\circ(1).
\end{equation}
In particular, by taking $b(x)=x,$ for $x\in\lbrack0,1]$, we get
\begin{equation}
\int_{0}^{1}(a-xa^{\prime})|(y_{n})_{x}|^{2}\,dx+\int_{0}^{1} \gamma_{n}^{2}\,|y_{n}%
(x)|^{2}\,dx-a(1)|(y_{n})_{x}(1)|^{2}=\circ(1),\label{eqcle00}%
\end{equation}
while for $b(x)=x-1,\forall\,x\in\lbrack0,1]$, we have
\begin{equation}
\int_{0}^{1}(a-(x-1)a^{\prime})|(y_{n})_{x}|^{2}\,dx+\int_{0}^{1}%
\gamma_{n}^{2} \,|y_{n}(x)|^{2}\,dx=\circ(1).\label{eqcle11}%
\end{equation}
Combining (\ref{eqcle00})-(\ref{eqcle11}), it follows that
\begin{equation}
\int_{0}^{1}a^{\prime}\,|(y_{n})_{x}|^{2}\,dx+a(1)|(y_{n})_{x}(1)|^{2}%
=\circ(1).\label{eqcle1}%
\end{equation}
Thus according to (\ref{1.33}), we obtain
\begin{equation}
(y_{n})_{x}(1)=\circ(1),\label{eqcle2}%
\end{equation}
which, together with (\ref{eqcle0}) and (\ref{eqcle1}), yields
\begin{equation}
y_{n}\rightarrow0\;\,\hbox{in}\;H^{1}(0,1).\label{eqcle3}%
\end{equation}

\noindent\textbf{Third step.}\medskip

Taking the inner product of \eqref{1.13bkv} with $\frac{z_{n}}{\gamma_{n}^{2}%
}$ in $L^{2}(0,1)$, we have
\[
i\gamma_{n}\int_{0}^{1}|z_{n}(x)|^{2}\,dx+\int_{0}^{1}(a(y_{n}))_{x}%
(x)\,\overline{i\gamma_{n}\,(y_{n})_{x}(x)-(f_{n})_{x}(x)}\,dx
\]%
\[
-a(1)(y_{n})_{x}(1)\,\overline{i\gamma_{n}y_{n}(1)-f_{n}(1)}+a(0)(y_{n}%
)_{x}(0)\,\overline{i\gamma_{n}y_{n}(0)-f_{n}(0)}%
\]%
\[
=\int_{0}^{1}g_{n}(x)\,\overline{i\gamma_{n}y_{n}(x)-f_{n}(x)}\,dx=\circ(1).
\]
Hence
\[
\int_{0}^{1}|z_{n}(x)|^{2}\,dx-\int_{0}^{1}a(x)\,|(y_{n})_{x}(x)|^{2}%
\,dx=\circ(1),
\]
which together with (\ref{eqcle3}) leads to 
\begin{equation}
z_{n}\rightarrow0\;\hbox{in}\;L^{2}(0,1).\label{cvz}%
\end{equation}
Lastly, the identities \eqref{unkv}, \eqref{znkv}, \eqref{cveta},
\eqref{eqcle3} and \eqref{cvz} clearly contradicts the fact that $\left\Vert
Z_{n}\right\Vert _{\dot{\mathcal{X}}}=1,\;\forall\ n\in\mathbb{N}.$\medskip

Thereby, the two assumptions of Theorem \ref{lemraokv} are proved and the
proof of Theorem \ref{lrkv} is thus completed.
\end{proof}

\begin{rem}
Combining Theorem \ref{t2} and Theorem \ref{lrkv}, one can claim that the
solutions of the closed-loop system (\ref{1.6}) polynomially tend in
${\mathcal{X}}$ to $(\Omega,0,0,0,0)$ as $t\longrightarrow+\infty$, where
$\Omega$ is given by (\ref{cste}). \label{pol}
\end{rem}

\section{Conclusions and discussions}

\label{sect6} To recapitulate, this work dealt with the analysis of overhead
system under the presence of a constant time-delay in the boundary velocity
control. Assuming that the feedback gain of the delayed term is small, it has
been shown that the system is well-posed whose proof is based on the
introduction of a suitable energy-norm. Additionally, it has been proved that
the solutions of the system asymptotically converge to an equilibrium state
which is explicitly given and depends on the initial conditions. The proof of
this result utilized the well-known LaSalle principle. More importantly, the
polynomial convergence of solutions has been obtained.

\medskip

We point out that there are many problems which could be treated. For instance,
it is quite natural to wonder whether the results obtained in this article
could be extended to the case where the control is nonlinear. Moreover, if the
delay occurring in the boundary control is time-dependent, then does the
convergence result still hold? This will be the focus of our attention in
future works.

\end{document}